\documentclass{amsart}
\usepackage{amsmath,amsfonts,amssymb,mathrsfs,amscd,comment}

\newcommand{\mb}[1]{{\textbf {\textit#1}}}

\newtheorem{theorem}{Theorem}
\newtheorem{proposition}{Proposition}

\begin{document}

\title[On manifolds defined by 4-colourings of simple 3-polytopes]{On manifolds defined by 4-colourings\\ of simple 3-polytopes}

\author[V. Buchstaber]{Victor Buchstaber}
\address{Steklov Mathematical Institute, Russian Academy of
Sciences, Gubkina str.~8, 119991 Moscow, RUSSIA}
\email{buchstab@mi.ras.ru}

\author[T. Panov]{Taras Panov}
\address{Department of Mathematics and Mechanics, Moscow
State University, Leninskie Gory, 119991 Moscow, RUSSIA,
\newline\indent Institute for Theoretical and Experimental Physics,
Moscow, \quad \emph{and}
\newline\indent Institute for Information Transmission Problems,
Russian Academy of Sciences} 
\email{tpanov@mech.math.msu.su}
\urladdr{http://higeom.math.msu.su/people/taras/}

\date{}


\maketitle

\makeatletter
\renewcommand{\@makefnmark}{}
\makeatother
\footnotetext{The work was carried out at the Steklov Institute of Mathematics and
supported by the Russian Science Foundation grant no.~14-11-00414.}

Let $\mathcal P$ be the class of combinatorial $3$-dimensional simple polytopes $P$, different from a tetrahedron, without $3$- and $4$-belts of facets. By the results of Pogorelov~\cite{pogo67} and Andreev~\cite{andr70}, a polytope $P$ admits a realisation in Lobachevsky space $\mathbb L^3$ with right dihedral angles if and only if $P\in\mathcal P$. We consider two families of smooth manifolds defined by regular $4$-colourings of polytopes $P\in\mathcal P$: six-dimensional quasitoric manifolds over $P$ and three-dimensional small covers of~$P$; the latter are also known as three-dimensional hyperbolic manifolds of L\"obell type~\cite{vesn87}. We prove that two manifolds from either of the families are diffeomorphic if and only if the corresponding $4$-colourings are equivalent.

A \emph{quasitoric manifold} (respectively, a \emph{small over}) over a simple $n$-polytope $P$ is a  $2n$-dimensional ($n$-dimensional) smooth manifold $M$ with a locally standard action of the torus~$T^n$ (the group $\mathbb Z_2^n$) and a projection $M\to P$ whose fibres are the orbits of the action, see~\cite{da-ja91},~\cite{bu-pa15}. 

Let $\mathcal F=\{F_1,\ldots,F_m\}$ be the set of facets of a simple $3$-polytope~$P$. A \emph{characteristic function} over $\mathbb Z$ (over $\mathbb Z_2$) is a map $\lambda\colon\mathcal F\to\mathbb Z^n$ ($\lambda\colon\mathcal F\to\mathbb Z_2^n$) satisfying the condition: if $F_{i_1}\cap F_{i_2}\cap F_{i_3}$ is a vertex, then $\lambda(F_{i_1}),\lambda(F_{i_2}),\lambda(F_{i_3})$ is a basis of the lattice~$\mathbb Z^n$ (of the space~$\mathbb Z_2^n$).  Characteristic functions $\lambda$ and $\lambda'$ are \emph{equivalent} ($\lambda\sim\lambda'$) if one is obtained from the other by a change of basis in~$\mathbb Z^n$ and changing the direction of some of the vectors $\lambda(F_i)$ to the opposite (by a change of basis in~$\mathbb Z_2^n$). \emph{Characteristic pairs} $(P,\lambda)$ and $(P',\lambda')$ are \emph{equivalent} ($(P,\lambda)\sim(P',\lambda')$) if $P$ and $P'$ are combinatorially equivalent ($P\simeq P'$) and $\lambda\sim\lambda'$.

Every quasitoric manifold (small cover) $M$ over $P$ is defined by a characteristic pair $(P,\lambda)$; with two such manifolds $M=M(P,\lambda)$ and $M'=M'(P',\lambda')$ being equivariantly homeomorphic if and only if $(P,\lambda)\sim(P',\lambda')$ (see~\cite{da-ja91}, \cite[Prop.~7.3.11]{bu-pa15}). In general, there exist non-equivalent pairs $(P,\lambda)$ and $(P',\lambda')$ whose corresponding manifolds $M$ and $M'$ are (non-equivariantly) diffeomorphic.

A (regular) \emph{$4$-colouring} of a simple polytope $P$ is a map $\chi\colon\mathcal F\to\{1,2,3,4\}$ such that $\chi(F_i)\ne\chi(F_j)$ whenever $F_i\cap F_j\ne\varnothing$. Such a $4$-colouring always exists by the Four Colour Theorem. Two $4$-colourings $\chi$ and $\chi'$ are \emph{equivalent} ($\chi\sim\chi'$) if $\chi'=\sigma\chi$ for a permutation~$\sigma\in S_4$.

Let $\chi$ be a $4$-colouring, $\mb a_1,\mb a_2,\mb a_3$ a basis in $\mathbb Z^3$, and $\varepsilon_i=\pm1$, $i=1,2,3$. These data define a characteristic function
$\lambda=\lambda(\chi,\mb a_1,\mb a_2,\mb a_3,\varepsilon_1,\varepsilon_2,\varepsilon_3)$ given by 
\[
  \lambda(F)=\begin{cases}
     \mb a_i&\text{if }\chi(F)=i,\quad i=1,2,3,\\
     \varepsilon_1\mb a_1+\varepsilon_2\mb a_2+\varepsilon_3\mb a_3    
  &\text{if }\chi(F)=4.
  \end{cases}
\]
Denote by $\mb e_1,\mb e_2,\mb e_3$ the standard basis $(1,0,0),(0,1,0),(0,0,1)$ in~$\mathbb Z^3$ or~$\mathbb Z_2^3$.

\begin{proposition} We have
$\lambda(\chi,\mb a_1,\mb a_2,\mb a_3,\varepsilon_1,\varepsilon_2,\varepsilon_3)\sim\lambda(\chi,\mb e_1,\mb e_2,\mb e_3,1,1,1)$.
\end{proposition}
\begin{proof}\sloppy
We have
$ \lambda(\chi,\mb a_1,\mb a_2,\mb a_3,\varepsilon_1,\varepsilon_2,\varepsilon_3)\sim
  \lambda(\chi,\mb e_1,\mb e_2,\mb e_3,\varepsilon_1,\varepsilon_2,\varepsilon_3)
  \sim\lambda(\chi,\varepsilon_1\mb e_1,\varepsilon_2\mb e_2,\varepsilon_3\mb e_3,1,1,1)\sim
  \lambda(\chi,\mb e_1,\mb e_2,\mb e_3,1,1,1)$,
where the first and third equivalences come from a change of basis in~$\mathbb Z^3$, while the second equivalence comes from a change of the direction of vectors.
\end{proof}

Note that the equivalence classes of characteristic functions are the orbits of the group $\mathrm{GL}_3(\mathbb Z)$ or $\mathrm{GL}_3(\mathbb Z_2)$, while the equivalence classes of $4$-colourings are the orbits of the symmetric group~$S_4$. Nevertheless, we have

\begin{proposition}
$\chi\sim\chi'\;\Leftrightarrow\;\lambda_\chi\sim\lambda_{\chi'}$, where
$\lambda_\chi:=\lambda(\chi,\mb e_1,\mb e_2,\mb e_3,1,1,1)$
\end{proposition}
\begin{proof}
Assume that $\chi'=\sigma\chi$, $\sigma\in S_4$. Denote $\mb e_4:=\mb e_1+\mb e_2+\mb e_3$. We have $\mb e_{\sigma(4)}=\varepsilon_1\mb e_{\sigma(1)}+\varepsilon_2\mb e_{\sigma(2)}+\varepsilon_3\mb e_{\sigma(3)}$ for some $\varepsilon_i=\pm1$.  Then
$\lambda_{\chi'}=\lambda(\chi',\mb e_1,\mb e_2,\mb e_3,1,1,1)=
  \lambda(\chi,\mb e_{\sigma(1)},\mb e_{\sigma(2)},\mb e_{\sigma(3)},\varepsilon_1,\varepsilon_2,
  \varepsilon_3)\sim\lambda_\chi$,
where the equivalence follows from Proposition~1.

\sloppy Conversely, assume $\lambda_\chi\sim\lambda_{\chi'}$. By the definition of characteristic functions, we have
$\lambda_{\chi'}=\lambda(\chi,\mb a_1,\mb a_2,\mb a_3,\varepsilon_1,\varepsilon_2,\varepsilon_3)$ for some basis $\mb a_1,\mb a_2,\mb a_3$ and $\varepsilon_i=\pm1$. The image of $\lambda_{\chi'}$ is the set $\{\mb e_1,\mb e_2,\mb e_3,\mb e_1+\mb e_2+\mb e_3\}$, while the image of $\lambda(\chi,\mb a_1,\mb a_2,\mb a_3,\varepsilon_1,\varepsilon_2,\varepsilon_3)$ is the set $\{\mb a_1,\mb a_2,\mb a_3$,     $\varepsilon_1\mb a_1+\varepsilon_2\mb a_2+\varepsilon_3\mb a_3\}$. Therefore, these two sets of $4$ vectors coincide, that is, $\mb a_1=\mb e_{\sigma(1)}$, $\mb a_2=\mb e_{\sigma(2)}$, $\mb a_3=\mb e_{\sigma(3)}$ and $\varepsilon_1\mb a_1+\varepsilon_2\mb a_2+\varepsilon_3\mb a_3=\mb e_{\sigma(4)}$ for some $\sigma\in S_4$. Thus, $\chi'=\sigma\chi$ and $\chi\sim\chi'$.
\end{proof}

\begin{theorem}[\cite{b-e-m-p-p}] 
Let $M=M(P,\lambda)$ and $M'=M'(P',\lambda')$ be quasitoric manifolds (or small covers), an assume that $P$ belongs to the class $\mathcal P$. Then $M$ and $M'$ are diffeomorphic if and only if the characteristic pairs $(P,\lambda)$ and $(P',\lambda')$ are equivalent.
\end{theorem}

\begin{theorem}[main result]
Assume given a polytope  $P\in\mathcal P$ with a $4$-colouring~$\chi$, and let $P'$ be another simple $3$-polytope with a $4$-colouring~$\chi'$. Then the $6$-\-di\-men\-si\-o\-nal quasitoric manifolds (or 
$3$-dimensional hyperbolic manifolds of L\"obell type) $M=M(P,\lambda_\chi)$ and $M'=M'(P',\lambda_{\chi'})$ are diffeomorphic if and only if $P\simeq P'$ and $\chi\sim\chi'$.
\end{theorem}
\begin{proof}
If $P\simeq P'$ and $\chi\sim\chi'$, then $\lambda_\chi\sim\lambda_{\chi'}$, by Proposition~2. Then the pairs $(P,\lambda_\chi)$ and $(P',\lambda_{\chi'})$ are equivalent, so $M$ and $M'$ are diffeomorphic.

Conversely, if $M$ and $M'$ are diffeomorphic, then $P\simeq P'$ and $\lambda_\chi\sim\lambda_{\chi'}$, by Theorem~1. Therefore, $\chi\sim\chi'$, by Proposition~2.
\end{proof}

We say that a characteristic function $\lambda\colon \mathcal F\to\mathbb Z^3$ \emph{is defined by a $4$-colouring~$\chi$} if $\lambda(F)=\lambda(F')$ whenever $\chi(F)=\chi(F')$. The image of such a characteristic function consists of $4$ vectors. Examples are $\lambda(\chi,\mb a_1,\mb a_2,\mb a_3,\varepsilon_1,\varepsilon_2,\varepsilon_3)$ and $\lambda_\chi$. A regular $4$-colouring of a simple $3$-polytope is \emph{complete} if for any set of three different colours there exists a vertex whose incident facets have these colours.

\begin{proposition}
Let $\chi$ be a complete $4$-colouring. Then any characteristic function $\lambda\colon \mathcal F\to\mathbb Z^3$ defined by~$\chi$ is equivalent to~$\lambda_\chi$.
\end{proposition}

Note that it is necessary that the $4$-colouring is complete. For example, assume that there is no vertex with the combination of colours $\{1,2,4\}$ for the incident facets. Then for each $k\in\mathbb Z$ consider the characteristic function $\lambda_{\chi,k}$ given by 
\[
  \lambda_{\chi,k}(F)=\begin{cases}
     \mb e_i&\text{if }\chi(F)=i,\quad i=1,2,3,\\
     \mb e_1+\mb e_2+k\mb e_3
     &\text{if }\chi(F)=4.
  \end{cases}
\]
Then $\lambda_{\chi,k}\not\sim\lambda_\chi$ for $k\ne\pm1$  (and $\lambda_{\chi,0}\not\sim\lambda_{\chi,1}=\lambda_\chi$ over $\mathbb Z_2$).

The class $\mathcal P$ contains all \emph{fullerenes}, that is, simple $3$-polytopes with only pentagonal and hexagonal facets~\cite{bu-er}. If a fullerene has two adjacent pentagons, then all four combinations of three colours are realised in the vertices of these two pentagons, so any $4$-colouring of such a fullerene is complete. For fullerenes without adjacent pentagons (\emph{IPR-fullerenes}) there exist non-complete $4$-colourings~$\chi$. It is easy to see that the corresponding hyperbolic $3$-manifolds $M(P,\lambda_{\chi,0})$ are non-orientable, unlike $M(P,\lambda_\chi)$.

\end{document}